\documentclass[12pt,reqno]{amsart}

\usepackage[all]{xy}
\usepackage{color}
\usepackage{amssymb, amsmath}
\usepackage{ulem}
\newcommand{\bbC}{{\mathbb C}}
\newcommand{\bbZ}{{\mathbb Z}}
\newcommand{\bbP}{{\mathbb P}}
\newcommand{\Oh}{{\mathcal O}}
\DeclareMathOperator{\Spec}{Spec}

\DeclareMathOperator{\CH}{CH}
\DeclareMathOperator{\HH}{H}
\DeclareMathOperator{\FF}{F}
\DeclareMathOperator{\RR}{R}
\DeclareMathOperator{\gr}{gr}
\newcommand{\onto}{\twoheadrightarrow}
\newcommand{\into}{\hookrightarrow}

\newcommand{\tensor}{\otimes}
\newcommand{\isom}{\cong}

\newtheorem{thm}{Theorem}[section]
\newtheorem{lemma}[thm]{Lemma}
\newtheorem{prop}[thm]{Proposition}
\newtheorem{cor}[thm]{Corollary}

\theoremstyle{definition}
\newtheorem{remark}[thm]{Remark}
\newtheorem{defn}[thm]{Definition}

\newenvironment{diagram}[1]{\arraycolsep=\doublerulesep\begin{array}{#1}
 }{\end{array}}
\long\def\comment#1{}

\numberwithin{equation}{section}

\begin{document}

\title{The Hodge rank of ACM bundles and Franchetta's conjecture}
\author[I. Biswas]{Indranil Biswas} 

\address{School of Mathematics, Tata Institute of Fundamental
Research, Homi Bhabha Road, Mumbai 400005, India}

\email{indranil@math.tifr.res.in}

\author[G. V. Ravindra]{G.~V.~Ravindra} 

\address{University of Missouri -- St. Louis, St. Louis, MO 63121, USA}

\email{girivarur@umsl.edu} 



\begin{abstract}
We prove that on a general hypersurface in ${\mathbb P}^N$ of degree $d$ and dimension at least $2$, if an 
arithmetically Cohen-Macaulay (ACM) bundle $E$ and its dual have small regularity, then any 
non-trivial Hodge class in $\HH^{n}(X,\, E\tensor\Omega^{n}_X)$\ , 
$n\,=\,\lfloor{\frac{\dim{X}}{2}}\rfloor$, produces a trivial direct summand of $E$. As a consequence, 
we prove that there is no universal Ulrich bundle on the family of smooth hypersurfaces of 
degree $d\geq 3$ and dimension at least $4$. This last statement may be viewed as a Franchetta-type 
conjecture for Ulrich bundles on smooth hypersurfaces.
\end{abstract}

\date{}

\maketitle

\section{Introduction}

The Franchetta conjecture, first stated in \cite{Franchetta} and proved in \cite{Harer} 
(see also \cite{AC}), says that for the universal family of genus $g\,\geq\, 2$ curves 
$$\pi\,:\, \mathcal{C}_g \,\longrightarrow\, \mathcal{M}_g,$$ the restriction of any line 
bundle $\mathcal{L}$ on $\mathcal{C}_g$ to any smooth fiber $C_s\,:=\,\pi^{-1}(s)$ is a power 
of the canonical bundle.

Since then, various analogues of Franchetta's conjectures have been posed. For instance, 
motivated by the work of Beauville and Voisin on the Chow groups of $K3$ surfaces (see 
\cite{BV}), O'Grady \cite{OG} conjectured that for a smooth family of $K3$ surfaces 
$\mathcal{X}\,\longrightarrow\, S$, the restriction map of rational Chow groups 
$\CH^2(\mathcal{X})\,\longrightarrow\, \CH^2(X)$, where $X$ is any smooth fiber, is generated by the class 
$\mathfrak{o}_X$ -- this is the class of a point lying on a rational curve $C$ in $X$. 
There are other such generalizations to Chow groups of higher codimensional cycles for the
hyper-K\"ahler varieties (see \cite{BL}).
 
In this article, we study a Franchetta type conjecture in a different direction, {\it viz.}
{\it for higher rank bundles} on smooth projective varieties.

Recall that a vector bundle $E$ on a polarized projective variety $(X, \,\Oh_X(1))$ is said 
to be {\it arithmetically Cohen-Macaulay} (ACM) if it has no intermediate cohomology, 
meaning $$\HH^i_\ast(X, \,E)\,:=\, \bigoplus_{\nu\in\bbZ}\HH^i(X,\, E(\nu))\,=\,0
\, \ \ \forall \ \, 0 \,< 
\,i\, <\, \dim{X}.$$ A result of Horrocks \cite{Horrocks} states that a vector bundle on $\bbP^N$ 
is a direct sum of line bundles if and only if it is ACM. This statement however does not hold for 
other classes of varieties in this generality (see for example Proposition $5$ in 
\cite{KRR2}). More recently, a special class of ACM bundles, namely that of 
{\it Ulrich bundles}, has received a considerable amount of interest. One striking result 
that has been proved is that every smooth complete intersection variety over an algebraically 
closed field supports an Ulrich bundle (see \cite{BHU}). We recall the following characterization
of Ulrich bundles.

\begin{prop}[see Theorem 2.3, \cite{Beauville-Ulrich}]
Let $X\,\subset \, \bbP^N$ be a smooth projective variety, and let $E$ be a vector bundle on $X$.
The following conditions are equivalent:
\begin{enumerate}
\item[(i)] There exists a linear resolution 
$$0 \,\longrightarrow\, \widetilde{L}_c\,\longrightarrow \, \widetilde{L}_{c-1}\,\longrightarrow\, \cdots \,
\longrightarrow \, \widetilde{L}_0 \, \longrightarrow \, E \,\longrightarrow \, 0,$$
where $c\,:=\,{\rm codim}\,(X,~ \bbP^N)$ and $\widetilde{F}_i\,=\,\Oh_{\bbP^N}(-i)^{\oplus\,b_i}$.

\item[(ii)] The cohomology $\HH^i(X,\, E(-p))\,=\,0$ for all $i \,>\, 0$ and $1\,\le \, p \,\leq\, \dim{X}$.

\item[(iii)] If $\pi\,:\, X\,\longrightarrow\, \bbP^{\dim{X}}$ is a finite linear projection, then $\pi_\ast(E)$ is the trivial bundle.
\end{enumerate}
\end{prop}

An ACM bundle $E$ on a smooth hypersurface $X \,\subset \,\bbP^{N}$ comes equipped with a unique minimal resolution 
\begin{equation}\label{minres}
0 \,\longrightarrow\, \widetilde{F}_1 \,\stackrel{\Phi}{\longrightarrow}\,
\widetilde{F}_0 \,\longrightarrow\, E \,\longrightarrow\, 0,
\end{equation}
where $\widetilde{F}_1$ and $\widetilde{F}_0$ are sums of line bundles on $\bbP^{N}$. Here
by {\it minimal}, we mean that no non-zero constant appears as an entry of the matrix
$\Phi$ in \eqref{minres}.
In this setting, an Ulrich bundle $E$ is one which admits a {\it linear resolution}:
$$0\,\longrightarrow\, \Oh_{\bbP^N}(-1)^M \, \stackrel{\Phi}{\longrightarrow}\,
\Oh_{\bbP^N}^M\,\longrightarrow\, E \,\longrightarrow\, 0.$$

It is known that $M\,=\,rd$, where $r$ is the rank of $E$ and $d$ is the degree of the hypersurface. It follows from the
resolution that $E$ is $0$-regular (in the sense of Castelnuovo and Mumford).
It is also known that the dual $E^\vee$ is $(d-1)$-regular, and in particular $h^0(E^\vee)\,=\,0$. 
We refer the reader to \cite{Beauville-Ulrich} and \cite{Coskun}
 for a quick introduction to Ulrich bundles, a discussion of their properties and their relevance.

Since any smooth complex hypersurface supports an Ulrich bundle \cite{BHU},
we may ask if these bundles are generic in the sense of 
the Franchetta-type conjectures. More precisely, one may ask:

\begin{quote}
Let $\mathcal{X} \,\longrightarrow\, S$ be the universal family of smooth hypersurfaces of
degree $d$ in $\bbP^N$. 
Is there a vector bundle $\mathcal{E}$ on $\mathcal{X}$ which is flat over the base $S$ such
that its restriction to any smooth fiber $X_s$, $s\,\in\, S$, is an Ulrich bundle?
\end{quote}

On the other hand, when a smooth family of varieties in a fixed ambient variety $P$, say 
$\mathcal{X} \,\subset \,S \times P$ is considered, it is known in various contexts that 
the behaviour of the generic fiber is related to that of $P$ (see \textsection~$2$, 
Proposition in \cite{Beauville} for a precise statement enunciating this principle). From 
this point of view, the Lefschetz theorems for Picard groups may be viewed as instances of 
the original Franchetta conjecture. Applying this principle to ACM bundles on a generic 
hypersurface (assuming that it holds) leads us to a version of Horrocks theorem.

For the sake of brevity, we adopt the following

\noindent{\bf Convention.} 
Let $p\,:\,\mathcal{X}\,\longrightarrow\, S$ be the universal family of smooth
hypersurfaces of degree $d$ in $\bbP^N$. Let $S' \,\longrightarrow\, S$ be an \'etale map and 
$\mathcal{E}$ on $\mathcal{X}'\,:=\,\mathcal{X}\times_S S'$ be a vector bundle which is flat over $S'$ such that for every $y\, \in\, S'$,
the bundle $\mathcal{E}_y$ on the hypersurface $X_y$ is ACM.
In what follows, the statement ``{\it $X$ is a general hypersurface of degree $d$ in $\bbP^N$ and $E$ is an ACM bundle on $X$}" will mean that
there exists an \'etale map $S' \to S$, a vector bundle $\mathcal{E}$ as above, and a point  $o\,\in\, S'$ such that  
$X\,=\,X_o$ and $E=\mathcal{E}_o$. 

Indeed, results proved in the recent past support the above expectation. In the foundational article 
on this subject \cite{Beauville-det}, it is shown that a general threefold of degree $d$ in 
$\bbP^4$ supports a rank $2$ Ulrich bundle if and only if $d\,\leq\, 5$. This was generalized 
to all rank $2$ indecomposable ACM bundles (see \cite{MPR2, R1}). The corresponding 
statements for rank $3$ bundles have been proved in \cite{AT2, RT5, RT6}. These results 
establish the base case of a version of Franchetta's conjecture that has been proposed for 
low rank bundles in \cite{RT5} where this is referred to as a {\it Noether-Lefschetz 
conjecture}. This conjecture in turn is inspired by a similar conjecture in \cite{BGS} (see 
\cite{Kleppe, MPR1, AT1, Erman} for progress related to this conjecture).
 
We prove the following generalization of these results.
 
\begin{thm}\label{ulrich}
There is no universal Ulrich bundle on smooth hypersurfaces of degree $d\,\geq\, 3$
and dimension $n\,\geq\, 4$.
\end{thm}

Theorem \ref{ulrich} follows from our main result below which uses a formalism
introduced in \cite{MS} to prove the Noether-Lefschetz theorem for Picard groups. 

%By the arguments of Theorem $3.4$ and Corollary $3.5$ of 
%\cite{MPR1}\footnote{This is only stated for rank two bundles, but the entire proof goes through
%with minor modifications for the general case.}, 

\begin{thm}\label{nltforacm_new}
Let $X$ be a general hypersurface in ${\mathbb P}^N$ of degree $d$ and dimension at least $2$,
and let $E$ be an ACM bundle on $X$. Let $n\,:=\,\lfloor\frac{\dim{X}}{2}\rfloor$ (so
that $\dim{X}\,=\,2n$ or $2n+1$). If the regularity of $E$ and its dual $E^\vee$ satisfy
$$m(E),\,m(E^\vee)\, \leq \, nd-2n-2,$$
then any non-zero Hodge class $\zeta\,\in\,\HH^n(X,\, E\tensor\Omega^n_X)$ gives
rise to a direct summand of $E$ isomorphic to
${\mathcal O}_X$. In particular, the rank of the Hodge cohomology group $\HH^n(X,\,
E\tensor\Omega^n_X)$ is bounded above by the rank of $E$.
\end{thm}

As a consequence, we obtain the following result of which Theorem \ref{ulrich} is a special case.

\begin{cor}\label{init}
Let $X$ be a general hypersurface in $\bbP^N$ of degree $d$ and dimension at least $4$, 
and let $E$ be an ACM bundle on $X$.
If $E$ is initialized, i.e., $h^0(E(-1))\,=\,0\,\neq\, h^0(E)$, 
and its regularity $m \,\leq \,nd-2n-2$, where $n\,:=\, \lfloor\frac{\dim{X}}{2} \rfloor$,
then $E$ is isomorphic to a direct sum of line bundles. 
\end{cor}

\begin{remark}\label{remin}
For an initialized ACM bundle $E$, we have $m(E) \,\geq\, 0$ and so
$nd-2n-2\,\geq\, 0$, or equivalently $d\,\geq\, 2+ \frac{2}{n}$. These degree bounds for 
odd-dimensional hypersurfaces are the same as the bounds that appear in
the Noether-Lefschetz theory (cf. the work of Green \cite{MGreen} and Voisin (unpublished) on
a conjecture of Griffiths and Harris generalizing the Noether-Lefschetz theorem).
\end{remark}

\begin{remark}\label{voisin}
Indecomposable ACM bundles on hypersurfaces are intimately related to Noether-Lefschetz 
theory in other ways as well. For example, a conjecture of Griffiths and Harris (see 
\cite{GH}) states that any curve $C$ in a general hypersurface $X\,\subset\,\bbP^4$ is an 
intersection of the form $X\cap{S}$ for some surface $S\,\subset\,\bbP^4$, thus generalizing 
the usual Noether-Lefschetz theorem for surfaces in $\bbP^3$. Voisin \cite{Voisin} 
constructs examples of curves disproving this conjecture. In \cite{KRR2}, its authors 
provide a general framework using ACM bundles to construct a large class of codimension $2$ 
subvarieties that do not arise as intersections and situates Voisin's examples amongst 
these.
\end{remark}

\noindent{\bf Acknowledgements:} 
The authors thank A.~Beauville for a very useful e-mail exchange during the preparation of the manuscript.
The authors thank the referee for not only a very careful reading of the manuscript but also for their comments 
and suggestions which have provided a greater clarity for the results presented here. The second author 
was partially supported by a grant from the Simons Foundation. 
\section{Preliminaries}

The base field $k$ is algebraically closed and its characteristic is zero.

Let $X \,\subset\, \bbP\,:=\,\bbP^{N}$ be an
irreducible smooth projective variety. Denote $W_1\,:=\, \HH^0(\bbP,\, \Oh_{\bbP}(1))$, and consider the dual of the Euler sequence
$$
0 \,\longrightarrow\, \Omega^1_\bbP \,\longrightarrow\, W_1^\ast\tensor\Oh_\bbP(-1)
\,\longrightarrow\,\Oh_\bbP\,\longrightarrow\, 0.
$$
It produces the Koszul complex
\begin{equation}\label{koszuleuler}
0 \,\longrightarrow\, \Omega_\bbP^{p} \,\longrightarrow\, \wedge^{p}W_1^\ast\tensor\Oh_\bbP(-p)
\,\longrightarrow\, \wedge^{p-1}W_1^\ast\tensor\Oh_\bbP(-p+1)
\end{equation}
$$
\cdots \,\longrightarrow\, \cdots \,\longrightarrow\, W_1^\ast\tensor\Oh_\bbP(-1)
\,\longrightarrow\, \Oh_\bbP \,\longrightarrow\, 0
$$
for every $p\, >\, 0$. Tensoring \eqref{koszuleuler} with any vector bundle $E$ on $X$ and
taking the corresponding long exact sequence of cohomologies, we obtain a coboundary map
\begin{equation}\label{partial}
\partial\,:\, \HH^0(X,\, E)\,\longrightarrow\, \HH^p(X, \,E\tensor\Omega^p_{\bbP}).
\end{equation}

The following splitting criterion is a very mild generalization of an elegant result we 
learnt from the papers of Arrondo-Malaspina \cite{AM} and Arrondo-Tocino \cite{AT}
and was the starting point for the present work.

\begin{lemma}\label{splitting_criterion}
Let $X\,\subset\,\bbP^N$ be a smooth projective variety with $\dim{X}\,=\,p+q$,
and let $E$ be a vector bundle on $X$ with $\HH^0(X,\, E)\, \not=\, 0$.
Assume that in the commutative diagram
\[
\begin{array}{ccc}
\HH^0(E) \times \HH^0(E^\vee) & \xrightarrow{\,\ \varphi\,\ } & \HH^0(\Oh_X) \\
 \Big\downarrow{\partial_1\times \partial_2} & & \Big\downarrow \partial\\
 \HH^p(E\tensor\Omega^p_\bbP) \times \HH^{q}(E^\vee\tensor\Omega^{q}_{\bbP}) 
 & \xrightarrow{\,\ \varphi_1\,\ } & \HH^{p+q}(\Omega^{p+q}_{\bbP_|X}) \\
\Big\downarrow{\rho_1\times \rho_2} & & \Big\downarrow\rho \\
 \HH^p(E\tensor\Omega^p_{X}) \times \HH^{q}(E^\vee\tensor\Omega^{q}_X) &
\xrightarrow{\,\ \varphi_2\,\ } & \HH^{\dim{X}}(\omega_{X}) 
\end{array}
\]
where $\partial_1$, $\partial_2$ and $\partial$ are as in \eqref{partial} while $\rho_1$, $\rho_2$ and $\rho$ are the
restriction maps, the maps
\begin{itemize}
\item $\partial_1$ and $\rho_1$ are injective, and
\item $\partial_2$ and $\rho_2$ are both surjective.
\end{itemize}
Then $\varphi$ is a non-zero pairing. Consequently, the trivial line bundle is a
direct summand of $E$.
\end{lemma}

\begin{proof}
Using contraction we have $(\Omega^p_{X})^\vee\otimes\omega_X\,=\, \Omega^{q}_X$, and hence
the map $\varphi_2$ in the diagram in the lemma is the Serre duality pairing. In particular,
$\varphi_2$ is a perfect pairing.

Take a non-zero section $s\,\in\, \HH^0(E)$. Then
$\rho_1\circ\partial_1(s) \,\neq\, 0$, as both $\rho_1$ and $\partial_1$ are given to be injective.
Since $\varphi_2$ is a perfect pairing, there 
exists a class $\xi \,\in\, \HH^{q}(E^\vee\tensor\Omega^p_X)$ such that
$\varphi_2((\rho_1\circ\partial_1(s))\otimes \xi) \,\neq\, 0$. Finally, since the maps $\partial_2$ and
$\rho_2$ are surjective, there is a section $s^\vee 
\,\in\, \HH^0(E^\vee)$ such that $\rho_2\circ\partial_2(s^\vee)\,=\,\xi$. By the commutativity of
the diagram we now have $\varphi(s\otimes s^\vee) \,\neq\, 0$.

Lemma \ref{split} shows that the trivial line bundle is a 
direct summand of $E$.
\end{proof}

\begin{lemma}\label{split}
Let $X$ be a smooth projective variety and $E$ a vector bundle on $X$. Assume there is a
non-zero pairing 
$$
\theta \,:\, \HH^0(E) \times \HH^0(E^\vee) \,\longrightarrow\, \HH^0(\Oh_X).$$
Then $E\,\isom\, E' \oplus \Oh_X$.
\end{lemma}

\begin{proof}
Take $s\,\in\, \HH^0(E)$ and $s^\vee\,\in\,\HH^0(E^\vee)$ such that
$\theta(s\otimes s^\vee) \,\neq\, 0$.
These give nonzero homomorphisms $s^\vee\,:\, E \,\longrightarrow\, \Oh_X $ and $s\,:\,\Oh_X
\,\longrightarrow\, E$ such that the composition of maps
$s^\vee\circ s\,:\, \Oh_X\,\longrightarrow\, \Oh_X$ is non-zero, and hence
$s^\vee\circ s$ is an isomorphism. This yields the required splitting.
\end{proof}

\begin{remark}
In practice, finding reasonably general conditions under which {\it both} $\rho_1$
and $\rho_2$ satisfy the hypotheses in Lemma \ref{splitting_criterion} can be challenging.
\end{remark}

\section{An infinitesimal Lefschetz theorem for Hodge classes of an ACM bundle}

As a prelude to our main result, we will establish, in Theorem \ref{inlt}, an {\it 
infinitesimal Lefschetz theorem} for Hodge classes of an ACM bundle on an even dimensional 
smooth projective hypersurface. The main tool here is a formalism, introduced by N. Mohan 
Kumar and V. Srinivas (unpublished), to prove the Noether-Lefschetz theorem for surfaces 
in $\bbP^3$. An expository account of this can be found in \cite{R}. Applications of this method 
are in \cite{KJ}, \cite{KRR2} and \cite{RS}.

We start by setting up some notation that will be employed in this section.
 
As before, $k$ is an algebraically closed field of characteristic zero.
Set $\bbP\,:=\,\bbP^{N}_k$ and
$W_d\,=\,\HH^0(\bbP,\,\Oh_{\bbP}(d))$. Let $S\,:=\,\bbP(W_d^{\ast})$ denote the
parameter space of all degree $d$ hypersurfaces in $\bbP$. We have the
short exact sequence
$$0\,\longrightarrow\,\mathcal{V}_d\,\longrightarrow\, W_d\tensor\Oh_{\bbP}\,=\,\HH^0(\bbP,\,
\Oh_{\bbP}(d))\tensor\Oh_{\bbP}\,\longrightarrow\, \Oh_{\bbP}(d) \,\longrightarrow\, 0,$$
where $\mathcal{V}_d$ is the kernel of the
evaluation map $W_d\tensor\Oh_{\bbP}\,\longrightarrow\, \Oh_{\bbP}(d)$. Then
\begin{equation}\label{cx}
\mathcal{X}\,:=\,\bbP(\mathcal{V}_d^{\ast})\,\longrightarrow\, S
\end{equation}
is the
universal family of all degree $d$ hypersurfaces. Let $X\,\subset\,\bbP$
be the smooth degree $d$ hypersurface corresponding to a closed point
$y\,\in\, S$. 

We have the standard exact sequence
$$0\,\longrightarrow\, \Oh_{\bbP}\,\longrightarrow\, \Oh_{\bbP}(d)
\,\longrightarrow\, \Oh_X(d) \,\longrightarrow\, 0.$$
The corresponding long exact sequence of cohomologies gives
$$ 0 \,\longrightarrow\, k \,\longrightarrow\, W_d \,\longrightarrow\, V \,\longrightarrow\, 0,$$
where
\begin{equation}\label{dv}
V\,:=\,\HH^0(X,\, \Oh_X(d)).
\end{equation}
Then $V \,\isom\, T_{S,y}$, the Zariski tangent space to
$S$ at the point $y$. Consider the standard decomposition
$$A\,:=\,\Oh_{S,y}/\mathfrak{m}_y^2\,=\,k\oplus
V^{\vee},$$ so that $\Omega^1_A\tensor{k}\,=\,
\HH^0(\Omega^1_{\Spec{A}})\,\isom\, V^{\vee}$. We have the inclusion maps
$$\{y\} \,\into \,\Spec{A} \,\into\, S\, ;$$
note that $\Spec{A}$ is the first order thickening of the smooth 
point $\{y\}\,\subset\, S$. Let
$$X_{\epsilon}\,:=\,\Spec{A}\times_S\mathcal{X}$$ 
denote the universal hypersurface over $\Spec{A}$, where
$\mathcal{X}$ is defined in \eqref{cx}. So $X_\epsilon$ is
the first order thickening of $X$ as a subscheme of $\mathcal{X}$.

Set $\bbP_A\,:=\,\bbP\times \Spec{A}$; let $p\,:\,\bbP_A\,\longrightarrow\, \bbP$
and $q\,:\,\bbP_A\,\longrightarrow\,\Spec{A}$ denote the two natural projections.
Consider the
conormal sheaf sequence for the inclusion map $\iota_{\epsilon}\,:\,X_{\epsilon}\,\into\,
\bbP_A$
$$\Oh_{X_{\epsilon}}(-d)\,\longrightarrow\, \Omega^1_{\bbP_A}\tensor
\Oh_{X_{\epsilon}}\,=\, p^{\ast}\Omega^1_{\bbP}\tensor\Oh_{X_{\epsilon}} \oplus
q^*\Omega^1_A\tensor\Oh_{X_{\epsilon}}\,\longrightarrow\, \Omega^1_{X_{\epsilon}}
\,\longrightarrow\,0.$$ 
Restricting this sequence to $X$ yields the following:
 $$\Oh_{X}(-d)\,\stackrel{(\alpha, \beta)}{\longrightarrow}\, \, \Omega^1_{\bbP}\tensor\Oh_{X} \oplus
V^\vee\tensor\Oh_{X}\,\stackrel{(\gamma, \delta)}\longrightarrow\, \Omega^1_{X_{\epsilon}}\tensor\Oh_X
\,\longrightarrow\,0.$$ 
Rewriting, we obtain the following
\begin{lemma}[{Mohan Kumar-Srinivas, \cite{MS}}]
There is a commutative diagram
\begin{equation}
\begin{diagram}{ccccccccc}\label{basicdgm}
0 & \longrightarrow & \Oh_X(-d) & \xrightarrow{\,\ \alpha\,\ } & \Omega^1_{\bbP}\tensor\Oh_X &
\longrightarrow & \Omega^1_X & \longrightarrow & 0\\
&&\,\,\,\Big\downarrow{-\beta} &&\,\,\,\Big\downarrow{\gamma} & &\Big\Vert &&\\
0 & \longrightarrow & V^{\vee}\tensor\Oh_X & \xrightarrow{\,\ \delta\,\ } & 
\Omega^1_{X_{\epsilon}}\tensor\Oh_X
& \longrightarrow & \Omega^1_X &\longrightarrow & 0\\
\end{diagram}
\end{equation}
where the top and bottom rows come from the inclusions $X\,\subset\,\bbP$ and
$X\,\subset \,X_{\epsilon}$ respectively, while the homomorphism $\beta$ is
the dual of the evaluation map.
\end{lemma}

Taking the $\ell$-th exterior power of \eqref{basicdgm} we obtain a commutative diagram
\begin{equation}
\begin{diagram}{ccccccccc}\label{cses}
0 & \longrightarrow & \Omega^{\ell-1}_X(-d) &\longrightarrow & \Omega^\ell_{\bbP}\tensor\Oh_X
&\longrightarrow & \Omega^\ell_X &\longrightarrow & 0 \\ 
& & \Big\downarrow & & \Big\downarrow & & \Big\Vert & & \\ 
0 & \longrightarrow & \Omega(\ell) &\longrightarrow & \Omega^\ell_{X_{\epsilon}}\tensor\Oh_X 
&\longrightarrow & \Omega^\ell_X &\longrightarrow & 0
\end{diagram}
\end{equation}
(see \cite[page 126, 5.16(d)]{Ha}), where $\Omega(\ell)$ comes
equipped with a decreasing filtration $\FF^{\ast}\Omega(\ell)$
for which
\begin{itemize}
\item $\FF^1(\Omega(\ell))\,=\,\Omega(\ell)$,
\item $\FF^{\ell+1}(\Omega(\ell))\,=\,0$, and
\item $\gr_{\FF}^j\Omega(\ell)\,:=\,\FF^j(\Omega(\ell))/\FF^{j+1}(\Omega(\ell))\,
=\, \Lambda^jV^{\vee}\tensor\Omega^{\ell-j}_X$ for all $j\,\geq\, 1$.
\end{itemize}

\begin{lemma}\label{omegacohomology}
For a smooth hypersurface $X\,\subset\,\bbP$ and an ACM bundle $E$ on $X$,
$$\HH^n(X,\,E\tensor\Omega^{n-j}_X)\,=\, 0\, = \, \HH^n(X,\,E(-d)\tensor\Omega^{n-j}_X)$$
for all $j\,\geq\, 1$, where $n\,:=\,\lfloor{\dim{X}/2}\rfloor$
(so $\dim{X}\,=\, 2n$ or $2n+1$).
\end{lemma}

\begin{proof}
Associated to the cotangent bundle sequence
\begin{equation}\label{ctbs}
0\,\longrightarrow\, \Oh_X(-d)\,\longrightarrow\, \Omega^1_{\bbP_|X}
\,\longrightarrow\, \Omega^1_X\,\longrightarrow\, 0
\end{equation}
we have the Koszul complex
\begin{equation}\label{koszulctbs}
0\,\longrightarrow\, \Oh_X(-(n-j)d)\,\longrightarrow\, \Omega^1_{\bbP|X}(-(n-j-1)d)
\,\longrightarrow\, \cdots
\end{equation}
$$
\cdots \,\longrightarrow\,\Omega^{n-j-1}_{\bbP|X}(-d)\,\longrightarrow\,
\Omega^{n-j}_{\bbP|X} \,\longrightarrow\, \Omega^{n-j}_X\,\longrightarrow\, 0.
$$

Tensoring \eqref{koszulctbs} with $E$ and taking the long exact sequence
of cohomologies, we see that it is enough to prove
the following for all $j\,\geq\, 1$:
\begin{enumerate}
\item[(a)] $\HH^{2n-j}(X,\, E(-(n-j)d)\,=\,0$, and
\item[(b)] $\HH^{n+a}(X,\, E\tensor\Omega^{n-j-a}_{\bbP_|X}(-ad))\,=\,0$ for $0\leq a < n-j$.
\end{enumerate}

Statement (a) is clear from the fact that $E$ is ACM, and $2n-j \,<\, \dim{X}
\,=\, 2n$ or $2n+1$.

Statement (b) can be verified as follows: Tensoring the minimal resolution
\eqref{minres} of $E$ with $\Omega^{n-j-a}_{\bbP}(-ad)$, we get that
\begin{equation}\label{minresOmega}
0\,\longrightarrow\, \widetilde{F}_1\tensor\Omega^{n-j-a}_{\bbP}(-ad) \,
\stackrel{\Phi}{\longrightarrow}\,
\widetilde{F}_0\tensor\Omega^{n-j-a}_{\bbP}(-ad)\,\longrightarrow\,
E\tensor\Omega^{n-j-a}_{\bbP}(-ad)\,\longrightarrow\, 0.
\end{equation}
The long exact sequence of cohomologies associated to
\eqref{minresOmega} produces the exact sequence
\begin{equation}\label{be}
\HH^{n+a}(\bbP,\, \widetilde{F}_0\tensor\Omega^{n-j-a}_{\bbP}(-ad))\,
\longrightarrow\, \HH^{n+a}(X,\, E\tensor\Omega^{n-j-a}_{\bbP_|X}(-ad))\,$$
$$ \longrightarrow\, \HH^{n+a+1}(\bbP,\, \widetilde{F}_1\tensor\Omega^{n-j-a}_{\bbP}(-ad)).
\end{equation}
Since $1\,\leq\, n-j-a\,\leq\, n-1$ and $n+a\, < \,2n-j\,< \,\dim{X}$, the two extreme terms
in \eqref{be} vanish by Bott's formula. The proof of the second vanishing also follows along the same lines.
\end{proof}

The following is a consequence of Lemma \ref{omegacohomology}.

\begin{cor}\label{Omegacohomology}
For a smooth, degree $d$ hypersurface $X\,\subset\,\bbP^{2n+1}$ or $X\, \subset \,\bbP^{2n+2}$ and an ACM bundle $E$ on $X$,
$$\HH^n(X,\,E\tensor\Omega(n))\,=\,0 \,= \, \HH^n(X,\,E(-d)\tensor\Omega(n)),$$ 
where $\Omega(n)$ is the filtered bundle in \eqref{cses}.
\end{cor}

\begin{proof}
Consider the exact sequences
\begin{equation}\label{ks1}
0\,\longrightarrow\, E\tensor{\FF}^{j+1}(\Omega(n))\,\longrightarrow\,
E\tensor{\FF}^j(\Omega(n)) \,\longrightarrow\,
E\tensor\gr_{\FF}^j\Omega
\end{equation}
$$
=\, E\tensor\Lambda^jV^{\ast}\tensor\Omega^{n-j}_X \,\longrightarrow\, 0.
$$
In view of Lemma \ref{omegacohomology}, the long exact sequence
of cohomologies associated to \eqref{ks1} produces a surjection
$$\HH^n(X,\,E\tensor{\FF}^{j+1}(\Omega(n))) \,\onto\,
\HH^n(X,\,E\tensor{\FF}^{j}(\Omega(n)))\ \ \, \forall \, \ j\,\geq\, 1.$$
Since 
$\HH^n(X,\,E\tensor\FF^{n}\Omega(n))\,=\,\HH^n(X,\,E\tensor\Lambda^nV^{\ast}\tensor\Oh_X)\,=\,0$,
we obtain the first vanishing. The proof of the second vanishing is along the same lines.
\end{proof}

Tensoring the Koszul complex
\begin{equation}\label{koszulctbs1}
0\,\longrightarrow\, \Oh_X(-(n-1)d)\,\longrightarrow\, \Omega^1_{\bbP|X}(-(n-2)d)
\,\longrightarrow\, \cdots
\end{equation}
$$
\cdots \,\longrightarrow\,\Omega^{n-2}_{\bbP|X}(-d)\,\longrightarrow\,
\Omega^{n-1}_{\bbP|X} \,\longrightarrow\, \Omega^{n-1}_X\,\longrightarrow\, 0.
$$
with $E(-d)$ and $E$ respectively, and taking the long exact sequence of cohomologies
we get coboundary maps
\begin{equation}\label{cob1}
\HH^{n+1}(X,\,E\tensor\Omega^{n-1}_X(-d))\,\longrightarrow\,\HH^{2n}(X,\,E(-nd)), 
\end{equation}
\begin{equation}\label{cob2}
\HH^{n+1}(X,\,E\tensor\Omega^{n-1}_X)\,\longrightarrow\, \HH^{2n}(X,\, E((-n+1)d)).
\end{equation}

\begin{lemma}\label{topcoh}
Let $X\,\subset\,\bbP^{2n+1}$, or $X\,\subset\, \bbP^{2n+2}$, 
be a smooth hypersurface of dimension at least two, and $E$ an ACM bundle on $X$. 
Then the coboundary maps in \eqref{cob1} and \eqref{cob2} are injections. 
\end{lemma}

\begin{proof}
We prove the case when $X$ has even dimension. The proof for odd dimensional hypersurfaces
of dimension at least three is identical.

When $n\,=\,1$, there is nothing to prove. So we assume that $n\geq 2$.
We argue as in the proof of Lemma \ref{omegacohomology}. 
From the Koszul complex \eqref{koszulctbs1} tensored with $E$ and $E(-d)$ respectively, 
it is enough to show that 
\begin{enumerate}
\item[(i)] $\HH^{n+1+a}(X,\, E\tensor\Omega^{n-1-a}_{\bbP_|X}(-ad))\,=\,0$ for $0\,\leq\,
a \,\leq \,n-2$, and
\item[(ii)] $\HH^{n+1+a}(X,\, E\tensor\Omega^{n-1-a}_{\bbP_|X}(-(a+1)d))\,=\,0$ for
$0\,\leq\, a \,\leq\, n-2$.
\end{enumerate}

Tensoring the minimal resolution
\eqref{minres} of $E$ with $\Omega^{n-1-a}_{\bbP}(-ad)$, and taking
the corresponding long exact sequence of cohomologies, we get the exact sequence
$$\HH^{n+1+a}(\bbP,\, \widetilde{F}_0\tensor\Omega^{n-1-a}_{\bbP}(-ad))\,
\longrightarrow\, \HH^{n+1+a}(X,\, E\tensor\Omega^{n-1-a}_{\bbP_|X}(-ad))\,$$
$$\longrightarrow\, \HH^{n+2+a}(\bbP,\, \widetilde{F}_1\tensor\Omega^{n-1-a}_{\bbP}(-ad)).$$
For the given values of $a$ and $n$, the two extreme terms vanish by Bott's formula.
This proves $(i)$. The proof of $(ii)$ is identical.
\end{proof}

\begin{prop}\label{rho1_odd}
Let $X$ be a smooth hypersurface and $E$ an ACM vector bundle on $X$.
Then the restriction map 
$$\rho\,:\, \HH^n(X,\, E\tensor\Omega^n_\bbP)\,\longrightarrow\, \HH^n(X,\, E\tensor\Omega^n_X)$$
is an injection when $\dim{X}\,=\,2n$, and an isomorphism when $\dim{X}\,=\,2n+1$.
\end{prop}

\begin{proof}
In view of the long exact sequence of cohomologies associated to the exact sequence
$$0\,\longrightarrow\, \Omega^{n-1}_X(-d)\,\longrightarrow\, \Omega^n_{\bbP}{_{|X}}
\,\longrightarrow\, \Omega^n_X \,\longrightarrow\, 0,$$
to prove the injectivity of $\rho$,
it suffices to show that $\HH^n(X,\, E\tensor\Omega^{n-1}_X(-d))\,=\,0$.
Now this follows from Lemma \ref{omegacohomology}.

The proof of surjectivity in the odd dimensional case is identical to the proof of Lemma \ref{topcoh}. 
The vanishings in this case yield the injection
$$\HH^{n+1}(X,\,E\tensor\Omega^{n-1}_X(-d))\,\into\,\HH^{2n}(X,\,E(-nd)).$$
However, since $\dim{X}=2n+1$, the cohomology on the right vanishes since $E$ is ACM. This gives us 
the desired result.
\end{proof}

In what follows, we shall say more about the image of the restriction map $\rho$ in the above
proposition in the even-dimensional case. For the sake of convenience, we make the
following definition:

\begin{defn}
Let $X$ be a smooth projective hypersurface of degree $d$, and $n = \lfloor{\frac{\dim X}{2}}\rfloor$.
We say a vector bundle $F$ on $X$
satisfies Hypothesis $(\star)$ if the natural homomorphism
\begin{equation}\label{regularity}
\HH^0(X,\, F(nd-2n-2))\tensor\HH^0(X,\, \Oh_X(d))
\,\longrightarrow\, \HH^0(X,\, F((n+1)d-2n-2))\tag{{$\star$}} 
\end{equation}
is surjective.
\end{defn}

\begin{remark}
Note that if $F$ is $(nd-2n-2)$--regular, then it satisfies Hypothesis $(\star)$.
\end{remark}

Consider the composition of maps of sheaves
$$\Omega^{n-1}_X\tensor\Oh_X(-d)\, =\, \Omega^{n-1}_X(-d)\,\longrightarrow\,
\Omega(n)\,\longrightarrow\, \Omega^{n-1}_X\tensor V^{\vee}\,=\, \gr^1_F\Omega(n),$$ 
where the first map is the leftmost map in diagram \eqref{cses} (set
$\ell\,=\,n$ in \eqref{cses}). This composition of maps coincides with the map
$1\tensor{-\beta}$ (see \eqref{basicdgm}). Let
\begin{equation}\label{phi}
\phi\,:\,\HH^{n+1}(X,\,E\tensor\Omega^{n-1}_X(-d))\,\longrightarrow\,
V^{\vee}\tensor\HH^{n+1}(X,\,E\tensor\Omega^{n-1}_X).
\end{equation}
be the map of cohomologies induced by the above map of sheaves.

\begin{lemma}\label{dgeq3}
Let $X\,\subset\,\bbP^{2n+1}$ be a smooth hypersurface and $E$ an ACM bundle on $X$
such that the dual bundle $E^\vee$ satisfies Hypothesis $(\star)$.
Then the map $\phi$ in \eqref{phi} is an injection.
\end{lemma}

\begin{proof}
Using Lemma \ref{topcoh}, we can identify $\phi$ with the
the restriction of the map 
$$\HH^{2n}(X,\,E\tensor\Oh_X(-nd))\,\longrightarrow\, V^{\vee}\tensor\HH^{2n}(X,\,E\tensor\Oh_X((-n+1)d)).$$
This is the dual of the cup product map
$$
\HH^0(X,\, E^\vee(nd-2n-2))\tensor\HH^0(X,\, \Oh_X(d))
\,\longrightarrow\, \HH^0(X,\, E^\vee((n+1)d-2n-2)).$$
By our hypothesis this cup product map is surjective. Hence the map $\phi$ is injective.
\end{proof}

\begin{prop}\label{evendim}
Let $X$ and $E$ be as in Lemma \ref{dgeq3}. Then there is an isomorphism
$$\HH^n(X,\,E\tensor\Omega^n_{\bbP}) \,\isom\,
\HH^n(X,\, E\tensor\Omega^n_{X_{\epsilon}}).$$
\end{prop}

\begin{proof}
Consider the cohomology diagram associated to \eqref{cses} for $\ell\,=\,n$:
\begin{equation}\label{impdgm}
\begin{matrix}
\HH^n(X,\,E\tensor\Omega^{n-1}_X(-d)) & \longrightarrow &
\HH^n(X,\,E\tensor\Omega^n_{\bbP}) & \longrightarrow &\\
\Big\downarrow && \Big\downarrow \\
\HH^n(X,\, E\tensor\Omega(n)) & \longrightarrow &
\HH^n(X,\, E\tensor\Omega^n_{X_{\epsilon}}) & \longrightarrow &
\end{matrix}
\end{equation}
$$
\begin{matrix}
\longrightarrow & \HH^n(X,\,E\tensor\Omega^n_X) & \longrightarrow &
\HH^{n+1}(X,\,E\tensor\Omega^{n-1}_X(-d)) \\ 
& \Big\Vert & & \Big\downarrow \\ 
\longrightarrow &
\HH^n(X,\,E\tensor\Omega^n_X) & \longrightarrow & \HH^{n+1}(X,\,E\tensor\Omega(n)).
\end{matrix}
$$
{}From Lemma \ref{omegacohomology} and Corollary \ref{Omegacohomology} it follows that 
$$\HH^n(X,\,E\tensor\Omega^{n-1}_X(-d))\,=\,\HH^n(X,\,E\tensor\Omega(n))\,=\,0.$$
Thus we see that the homomorphism
$$\HH^n(X,\,E\tensor\Omega^n_{\bbP}) \,\longrightarrow\,
\HH^n(X,\,E\tensor\Omega^n_{X_{\epsilon}})$$ is injective. 

To prove surjectivity, it suffices to show that the right most vertical homomorphism
in \eqref{impdgm}, i.e., the map
$\HH^{n+1}(X,\,E\tensor\Omega^{n-1}_X(-d))\,\longrightarrow\, \HH^{n+1}(X,\,E\tensor\Omega(n))$
is injective. 
By Lemma \ref{dgeq3}, the map $\phi$, which is the composition of maps
$$\HH^{n+1}(X,\, E\tensor\Omega^{n-1}_X(-d))\,\longrightarrow\, \HH^{n+1}(X,\,
E\tensor\Omega(n)) \,\longrightarrow\,
V^{\vee}\tensor\HH^{n+1}(X,\,E\tensor\Omega^{n-1}_X),$$ is injective. Therefore, the {\it first} map
$\HH^{n+1}(X,\,E\tensor\Omega^{n-1}_X(-d))\,\longrightarrow\, \HH^{n+1}(X,\,E\tensor\Omega(n))$
is injective.
\end{proof}

\begin{thm}[The infinitesimal Lefschetz theorem for Hodge classes]\label{inlt}
Let $X\,\subset\,\bbP^{2n+1}$ be a smooth hypersurface and $E$ an ACM bundle on $X$ such that its dual satisfies
the Hypothesis $(\star)$. Then there exists an exact sequence
$$
0\,\longrightarrow\, \HH^n(X,\,{E}\tensor\Omega^n_{X_{\epsilon}}{|_X})
\,\longrightarrow\,\HH^n(X,\,E\tensor\Omega^n_X)
\,\stackrel{\kappa}{\longrightarrow}\, V^{\vee}\tensor\HH^{n+1}(X,\,E\tensor\Omega^{n-1}_X).$$
\end{thm}

\begin{proof} The map $\kappa$ is just the composition of maps 
$$
\HH^n(X,\,E\tensor\Omega^n_X)\,\longrightarrow\, \HH^{n+1}(X,\,E\tensor\Omega(n))
\,\longrightarrow\, V^\vee\tensor\HH^{n+1}(X,\,E\tensor\Omega^{n-1}_X).$$
The theorem now follows from the identification in Proposition \ref{evendim}.
\end{proof}

\begin{remark}\label{ksmap}
The homomorphism $\kappa$ in Theorem \ref{inlt} is in fact the Kodaira-Spencer map
\begin{equation}
\begin{array}{ccc}
\HH^n(X,\,E\tensor\Omega^n_X) & \longrightarrow & \mbox{Hom}\left(T_{S,o},\, \HH^{n+1}(X,\,E\tensor\Omega^{n-1}_X)\right) \\
\xi & \longmapsto & (\frac{\partial}{\partial{x}} \longmapsto \frac{\partial}{\partial{x}}(\xi)) .
\end{array}
\end{equation}
Hence, if $X$ is general fiber in the universal family of degree $d$ hypersurfaces $\mathcal{X}\,\longrightarrow\, S$, so that $\xi$ is the restriction of a section in 
$\HH^0(S',\, \RR^n p_\ast(\mathcal{E}\tensor\Omega^n_{\mathcal{X'/S'}}))$, then $\xi\,\longmapsto\, 0$ under the Kodaira-Spencer map.
In particular, when $E\,\isom\,\Oh_X$, the cohomology group $\HH^n(X,\,\Omega^n_{X_{\epsilon}}{_{|X}})$ can be identified with the
subspace of cohomology classes of type $(n,\,n)$ on $X$ which deform
infinitesimally, i.e., they continue to be of type $(n,\,n)$ under infinitesimal
deformations. This statement is the usual {\it infinitesimal Noether-Lefschetz theorem} (see \cite{R}). 
\end{remark}

\section{Proof of Theorem \ref{nltforacm_new}}

We now apply the splitting criterion in Lemma \ref{splitting_criterion}
to the case of interest to us where $E$ is an ACM bundle on a general
hypersurface of dimension at least {\it two}.

\begin{lemma}\label{partial2}
Let $X\,\subset\,\bbP^{N}$ be an irreducible smooth projective variety 
and let $F$ be a vector bundle on $X$.
Set $n\,:=\,\lfloor{\frac{\dim X}{2}}\rfloor$, so that $\dim{X}$ is either $2n$ or $2n+1$ . 
Then the map
$$\partial\,:\, \HH^0(X,\, F) \,\longrightarrow\, \HH^n(X,\, F\tensor\Omega^n_\bbP)$$
in \eqref{partial} is surjective, if
$$
\HH^i(X,\, F(-i))\,=\,0
$$
for all $1\,\leq\,i \,\leq\, n$.
\end{lemma}

\begin{proof}
Tensor the Koszul complex \eqref{koszuleuler}
with $F$. Breaking up the resulting exact sequence into short exact sequences,
we get
$$
0\,\longrightarrow\, F\otimes \Omega^{k+1}_{\bbP}\,\longrightarrow\, F\otimes\wedge^{k+1}
W_1^\vee\tensor\Oh_{\bbP}(-k-1)\,\longrightarrow\,F\otimes\Omega^{k}_{\bbP}\,\longrightarrow\, 0
$$
for all $0\,\leq\, k \,\leq\, n-1$.
Consider the corresponding long exact sequence of cohomologies. Incorporating the
hypotheses, a surjection
$$\tilde\partial_k\,:\, \HH^k(X,\, F\tensor\Omega^k_{\bbP})
\,\longrightarrow\, \HH^{k+1}(X,\, F\tensor\Omega^{k+1}_{\bbP})$$
is obtained for every $0 \,\leq\, k \,\leq\, {n-1}$. Therefore, the
composition of homomorphisms
$$
\partial\,=\, \tilde\partial_{n-1}\circ\cdots\circ\tilde \partial_0 \,:\,
\HH^0(X,\, F)\,\longrightarrow\, \HH^{n}(X,\, F\tensor\Omega^{n}_{\bbP})
$$
is surjective.
\end{proof}

\begin{cor}\label{init2}
In the statement of Lemma \ref{partial2}, if we assume in addition that 
$$\HH^i(X,\, F(-i-1))\,=\,0, \,\, \,\,0 \,\leq\, i \,<\, n,$$ then the map $\partial$ is an isomorphism.
\end{cor}

\subsection{The case of even-dimensional hypersurfaces}

\begin{thm}\label{evencase}
Let $X\,\subset\,\bbP^{2n+1}$ 
be a general hypersurface of degree $d$ and dimension at least $2$,
and let $E$ be an ACM bundle on $X$ such that both $E$ and its dual $E^\vee$ satisfy Hypothesis \eqref{regularity}.
Then any non-zero class $\zeta \,\in\, \HH^n(X,\, E\tensor\Omega^n_X)$ yields a
direct summand of $E$ isomorphic to ${\mathcal O}_X$.
\end{thm}

\begin{proof}
We recall our convention in \textsection $1$.
Let $p\,:\,\mathcal{X}\,\longrightarrow\, S$ be the universal family of smooth
hypersurfaces of degree $d$ in $\bbP\,=\, \bbP^{2n+1}$, and let $q\,:\,\mathcal{X}
\,\longrightarrow\, \bbP$ be the natural projection. Then there exists an \'etale map $S' \,\longrightarrow\, S$ and 
an ACM bundle $\mathcal{E}$ on $\mathcal{X}'\,:=\,\mathcal{X}\times_S S'$ which is flat over $S'$ such that for every $y\, \in\, S'$,
the ACM bundle $\mathcal{E}_y$ on the hypersurface $X_y$ satisfies the hypothesis
in Lemma \ref{dgeq3}.

Let $o\,\in\, S'$ be a point such that  $X\,=\,X_0$.
Consider the diagram
\begin{equation}\label{va}
\begin{array}{ccccc}
X & \into & X_\epsilon & \into & \mathcal{X}' \\
\Big\downarrow & & \Big\downarrow & & \Big\downarrow \\
\Spec\bbC & \into & \Spec{A} & \into & S' \\
\end{array}
\end{equation}
With an abuse of notation, all the vertical maps in \eqref{va}
are denoted by $p$. By shrinking $S'$ if necessary, we have 
$$\HH^0(\mathcal{X}', \,\mathcal{E})\,\isom\, \HH^0(S',\, p_\ast\mathcal{E})
\,\isom\, (p_\ast\mathcal{E})_o\,\isom\, \HH^0(X,\, E)\,\stackrel{\partial_1}{\onto}
\,\HH^n(X,\, E\tensor\Omega^n_{\bbP}) 
\,\stackrel{\rho_1}{\into}\, \HH^n(X,\, E\tensor\Omega^n_{X}).$$
Here $\partial_1$ is a surjection by Lemma \ref{partial2} and $\rho_1$ is an injection
by Proposition \ref{rho1_odd}. A quick note on the isomorphisms above: 
the first isomorphism follows from a Leray spectral sequence argument, whereas the
second one follows since $p_\ast\mathcal{E}$ is locally trivial and we are allowed to shrink
$S'$ if necessary. The third isomorphism follows from the proper base change formula.

Since $E$ is a bundle on a general hypersurface $X$, once again, by the proper base change formula,
we have the identification
$$ (\RR^np_\ast(\mathcal{E}^\vee\tensor\Omega^n_{\mathcal{X}'/S'}))_o \,\isom \, 
\HH^n(X,\, E\tensor\Omega^n_{X}),$$
under which the Hodge class $\zeta$ comes from a global class. Consequently (see Remark \ref{ksmap}), 
it is in the kernel of
$$\kappa\,:\, \HH^n(X,\, E\tensor\Omega^n_X)
\,\longrightarrow\, V^\vee\tensor\HH^{n+1}(X, \,E\tensor\Omega^{n-1}_X).$$ 
Since $E^\vee$ satisfies Hypothesis $(\star)$, by Theorem \ref{inlt} and Proposition \ref{evendim},
$$\zeta\, \in \, \HH^n(X,\, E\tensor\Omega^n_{\bbP}).$$
Let
\begin{equation}\label{ts1}
\widetilde{s}\,\in \,\HH^0(\mathcal{X}',\, \mathcal{E}),
\end{equation}
with
\begin{equation}\label{ts2}
s\,\in\, \HH^0(X,\,E)
\end{equation}
its restriction to the fiber $X$, such that
$\partial_1(s)\,=\,\zeta$. Then, via the perfect pairing $\varphi_2$ in
Lemma \ref{splitting_criterion},
$$ \HH^n(X,\ E\tensor\Omega^n_{X}) \times \HH^{n}(X,\ E^\vee\tensor\Omega^{n}_X)
\,\xrightarrow{\,\ \varphi_2\,\ } \,\HH^{2n}(X,\ \omega_{X}) $$
its image
$\rho_1(\partial_1{s})\,\in\, \HH^n(X,\, E\tensor\Omega^n_X)$ defines a non-zero map 
$$\zeta\,:\, \HH^n(X,\, E^\vee\tensor\Omega^n_X) \,\longrightarrow\, \bbC,\ 
\ \hspace{2mm} \xi \,\longmapsto\, \varphi_2(\zeta,\,\xi) \,\in\, \bbC.$$

This perfect pairing is obtained by taking the fiber at the point $o \in S'$ of the perfect pairing (\,=\, Serre
duality in families )
$$ \RR^np_\ast(\mathcal{E}\tensor\Omega^n_{\mathcal{X}'/S'}) \,\times \,
 \RR^np_\ast(\mathcal{E}^\vee\tensor\Omega^n_{\mathcal{X}'/S'})
\,\longrightarrow\, \RR^{2n}p_\ast\omega_{\mathcal{X'}/S'}
\,\stackrel{tr}{\longrightarrow}\, \Oh_{S'}.$$
Here $tr$ is trace map.

Thus, we have a homomorphism
$$\widetilde{\zeta}\,:\, \RR^np_\ast(\mathcal{E}^\vee\tensor\Omega^n_{\mathcal{X}'/S'})
\,\longrightarrow\, \widetilde\bbC.$$
This $\widetilde{\zeta}$ is non-zero at the point $o\,\in\, S'$. Consequently, there is a
neighbourhood $U \,\subset\, S'$ of $o$ on which $\widetilde{\zeta}$
is non-zero everywhere. This implies that there is a non-zero section $$\widetilde\xi
\, \in\, \HH^0(U,\, \RR^np_\ast(\mathcal{E}^\vee\tensor\Omega^n_{\mathcal{X}'/S'}))$$
such that $(\rho_1(\partial_1\widetilde{s}),\, \widetilde\xi)$ is a non-zero section of
$\widetilde\bbC$ (see \eqref{ts1} for $\widetilde{s}$). Since $\widetilde\xi$ is defined over a neighbourhood
$U$, this implies that it {\it deforms infinitesimally}, and hence at the point $o\,\in\, U$
we have $$\xi_o:=\widetilde{\xi}(0)\, \in\, {\rm kernel}(\kappa).$$ 
As $E$ satisfies Hypothesis $(\star)$, by Lemma \ref{partial2}, and
the identifications in Theorem \ref{inlt} and Proposition \ref{evendim}, we have
$$
\HH^0(X,\, E^\vee)\,\stackrel{\partial_2}{\onto}
\,\HH^n(X,\, E^\vee\tensor\Omega^n_{\bbP}) 
\,\stackrel{\rho_2}{\into}\, \HH^n(X,\, E^\vee\tensor\Omega^n_{X}).$$
Thus, 
there exists $t\,\in\, \HH^0(X,\,E^\vee)$ such that $\xi_o\,=\,\rho_2(\partial_2(t))$, and in particular we see
that $(s,\,t)\,\neq\, 0$ (see \eqref{ts2} for $s$) under the pairing
$$\HH^0(X,\, E)\times\HH^0(X,\,E^\vee)\,\longrightarrow\, \HH^0(X,\, \Oh_X).$$
This implies that $E\,\isom\, E'\oplus \Oh_X$.
\end{proof}

\subsection{Odd dimensional hypersurface}

We begin with the following lemma.

\begin{lemma}\label{restrictionhodgeclass}
Let $X$ be a smooth hypersurface in $\bbP^{2n+2}$.
Suppose that $E$ is an ACM bundle with a section $s$ such that its image 
$$\xi\, :=\, \rho\circ \partial(s) \,\in \,\HH^n(X,\, E\tensor\Omega^n_X)$$
is non-zero (see Lemma \ref{partial2} and Proposition \ref{rho1_odd} for $\partial$ and $\rho$
respectively). Then there is a smooth hyperplane section $Y\,\subset\, X$
with defining polynomial $\ell \,\in \,\HH^0(X,\, \Oh_X(1))$ such that
the image of the restriction $s_{Y}$ in $\HH^n(Y,\,
E\tensor\Omega^n_Y)$, denoted by $\xi_Y$, is also non-zero. 
\end{lemma}

\begin{proof}
First note that we have a series of isomorphisms
$$\HH^i(X,\, E\tensor\Omega^i_\bbP)\,\isom \,\HH^{i+1}(X,\, E\tensor\Omega^{i+1}_\bbP),
\ \,\, 1\,\leq\, i \,\leq\, n-1.$$
Let $Y$ be a general hyperplane section and let $E_1:=E\tensor\Oh_Y$. We have a commutative diagram obtained from the Euler sequence
 \[
 \begin{array}{ccccccccc}
 &  &         0         &                          &        0          &     &        0            &  & \\    
      &      & \downarrow &                          & \downarrow &    & \downarrow  &  & \\
  & \to & W_1\tensor\HH^0(X,\, E(-2)) & \to & \HH^0(X,\,E(-1)) & \onto & \HH^1(X,\, E(-1)\tensor\Omega^1_{\bbP^{2n+2}}) & \to & 0\\
     &      & \downarrow &                          & \downarrow\times\ell &    & \downarrow  &  & \\
 & \to & W_1\tensor\HH^0(X,\, E(-1)) & \to & \HH^0(X,E) & \onto & \HH^1(X,\, E\tensor\Omega^1_{\bbP^{2n+2}}) & \to & 0\\
     &      & \downarrow &                          & \downarrow &    & \downarrow  &  & \\
 & \to & W_1\tensor\HH^0(X, \,E_1(-1)) & \to & \HH^0(Y,\,E_1) & \onto & \HH^1(Y,\, E\tensor\Omega^1_{\bbP^{2n+2}}\tensor\Oh_Y) & \to & 0 \\
     &      & \downarrow &                          & \downarrow &    & \downarrow  &  & \\
      &      &         0         &                          &        0          &     &        0            &  & \\    
 \end{array}
 \]
Further assume that $Y$ is such that the restriction $s_Y\,\neq\, 0$.
Suppose that $$s_Y \,\longmapsto\, \xi_Y\,=\,0$$ in the bottom row. This means that $s_Y$ is 
the restriction of $$\sum_i x_is_i \,\in\, W_1\tensor\HH^0(X,\,E(-1)),$$ 
where $x_i$ are the homogeneous coordinates.
Thus, we see that $$s -\sum_i x_is_i \,=\,\ell\cdot t$$ for some $t
\,\in\, \HH^0(X,\, E(-1))$, where $s$ is the section in the statement of the lemma. 
In particular, $s \,=\,\sum_i x_is'_i$, and 
hence the image of $s$ in $\HH^1(X,\, E\tensor\Omega^1_{\bbP^{2n+2}})$ is zero.
This is a contradiction. 
\end{proof}

\begin{thm}\label{oddcase}
Let $X\,\subset\,\bbP^{2n+2}$ be a general hypersurface of degree $d$ and dimension at least $3$,
and let $E$ be an ACM bundle on $X$ such that both $E$ and its dual $E^\vee$ satisfies 
Hypothesis \eqref{regularity}. Then any non-zero class $\xi\,\in\,\HH^n(X,\,E\tensor\Omega^n_X)$ yields 
a direct summand of $E$ isomorphic to ${\mathcal O}_X$.
\end{thm}

\begin{proof}
Since $X\,\subset\, \bbP^{2n+2}$ is a general hypersurface of degree $d$, and dimension at least $3$, 
a general hyperplane section of $X$ will yield a general hypersurface $Y$ of degree $d$ in $\bbP^{2n+1}$.
As in the statement, $E$ and $E^\vee$ are ACM vector bundles on $X$ satisfying Hypothesis $(\star)$. 
Then the restrictions $E_1\,:=\,E_{|Y}$ and $E_1^\vee$ are also ACM bundles on $Y$ satisfying Hypothesis 
$(\star)$.

Consider the sequence of maps
$$\HH^0(X,\,E) \,\stackrel{\partial}{\longrightarrow}\, \HH^n(X,\, E\tensor
\Omega^n_{\bbP^{2n+2}}) \,\stackrel{\rho}{\longrightarrow}\, \HH^n(X,\, E\tensor\Omega^n_X).$$
By Proposition \ref{rho1_odd}, the map $\rho$ is an isomorphism, and by Lemma \ref{partial2}, the map $\partial$
is a surjection. Hence there exists a non-zero section $s \,\in\, \HH^0(X,\, E)$ such that
$$\rho\circ\partial(s)\,=\, \xi.$$
Furthermore, by Lemma \ref{restrictionhodgeclass}, $\xi_Y$ is a non-zero Hodge class and consequently, 
by Theorem \ref{evencase}, the vector bundle $E_1$ splits off a
trivial rank one summand.

We claim that this splitting lifts to $X$ as well. To prove this, note that any map 
$\Oh_Y \,\longrightarrow\, E_1$ (respectively, $E_1 \,\longrightarrow\, \Oh_Y$) lifts to 
a map $\Oh_X \,\longrightarrow\, E$ (respectively, $E\,\longrightarrow\, \Oh_X$) because 
$E$ and $E^\vee$ are ACM (the restriction homomorphism $\HH^0(X,\,F)\,\longrightarrow\, 
\HH^0(Y,\, F_Y)$ is surjective for any ACM bundle $F$ on $X$). Now
the composition of maps $\Oh_X\,\longrightarrow\, E\,\longrightarrow\,\Oh_X$ is nowhere
zero since its restriction to $Y$ is nowhere zero. Hence $E$ has a direct summand
isomorphic to ${\mathcal O}_X$.
\end{proof}

\subsection{Proofs of Theorems}

\begin{proof}[Proof of Theorem \ref{ulrich}]
The proof is by contradiction.
Let $X$ be a general degree $d$ hypersurface of dimension at least four and 
suppose that $X$ supports an Ulrich bundle $E$.
Since $E$ is Ulrich, we have that
\begin{enumerate}
\item[(i)] $E$ is initialized and $0$-regular,

\item[(ii)] $E^\vee$ is $(d-1)$-regular, and

\item[(iii)] $h^0(E^\vee)\,=\,0$.
\end{enumerate}

It follows from these properties, and the exact sequence
$$0\,\longrightarrow\, E(-1)\,\longrightarrow\, E\,\longrightarrow\, E_{|Y}
\,\longrightarrow\, 0,$$
that the restriction of an Ulrich bundle on $X$ to a hyperplane section
$Y \,\subset\, X$ is also Ulrich (see \textsection~$3$, \cite{Beauville-Ulrich}).
Hence, without loss of generality, we may assume that $X$ is 
even-dimensional, namely $\dim{X}\,=\,2n\,\geq\, 4$.

Since $E$ is initialized, by Corollary \ref{init2}, there exists a non-zero Hodge cycle $\xi \,\in\, 
\HH^n(X,\, E\tensor\Omega^n_X)$, and by Serre duality a non-zero cycle $\xi^\perp \,\in\, 
\HH^n(X,\, E^\vee\tensor\Omega^n_X)$. The arguments in the proof of Theorem \ref{evencase}
show that $\xi^\perp$ deforms infinitesimally and hence $\xi^\perp\,\in\, \ker(\kappa)$.

Since $E$ is $0$-regular, this means that $E$ satisfies 
Hypothesis $(\star)$ and so by Theorem \ref{inlt} (applied to the (ACM) bundle $E^\vee$) and the surjectivity of the map 
$\partial$ (Lemma \ref{partial2}), we see that $h^0(E^\vee)\,\neq \,0$. This is a 
contradiction.
\end{proof}

\begin{proof}[Proof of Theorem \ref{nltforacm_new}]
By definition, if a bundle is $a$-regular, then it is also $(a+i)-$regular for any $i\,
\geq\, 0$. Therefore, $E$ and $E^\vee$ satisfy Hypothesis $(\star)$. The result now
follows from Theorem \ref{evencase} and Theorem \ref{oddcase}.
\end{proof}
 
\begin{proof}[Proof of Corollary \ref{init}]
We note that the dual of any initialized ACM bundle is $(d-1)$-regular,
and so both $E$ and $E^\vee$ satisfy Hypothesis $(\star)$.
Furthermore, we have an isomorphism 
$$\HH^0(X,\, E) \,\isom\, \HH^n(X, \,E\tensor\Omega^n_{\bbP}),$$ where 
$n\,=\,\lfloor\dim{X}/2\rfloor$ (Corollary \ref{init2}). Hence $E$ has non-zero $(n,\, n)$
Hodge cycles. Arguing as above, we see that $E\,\isom\,E'\oplus \Oh_X^{h^0(E)}$.

Let $a\,\geq\, 0$ be such that $E'_a\,:=\,E'(a)$ is initialized. Then $m(E'_a)\,=\,m(E')-a
\,\leq\,m(E)$. Hence $E'_a$ satisfies the hypothesis and thus we have $E'_a\,\isom\, E''\oplus 
\Oh_X^{h^0(E'_a)}$, or equivalently, $E'$ has a direct summand of the form 
$\Oh_X(-a)^{h^0(E'_a)}$.

Continuing in this fashion, we see that $E$ decomposes into a sum of line bundles.
 \end{proof}

\end{document}